\documentclass{amsart}
\usepackage[utf8]{inputenc}
\usepackage[english]{babel}
\usepackage{amsmath}
\usepackage{amsfonts}
\usepackage{amssymb}
\usepackage{amsthm}
\usepackage{cmap}
\usepackage{tikz-cd}
\usepackage[mathscr]{euscript}
\usepackage{hyperref}

\newtheorem{prop}{Proposition}
\newtheorem*{prop*}{Proposition}

\newtheorem{thm}{Theorem}

\newtheorem*{thm*}{Theorem}

\newtheorem{cor}{Corollary}
\theoremstyle{definition}

\newtheorem*{ntt*}{Notation}
\theoremstyle{remark}
\newtheorem{rmk}{Remark}

\DeclareMathOperator{\Q}{\mathbb Q}
\DeclareMathOperator{\Z}{\mathbb Z}

\DeclareMathOperator{\PP}{\mathbb P}
\DeclareMathOperator{\Co}{\mathbb C}
\DeclareMathOperator{\Oh}{\mathcal O}
\DeclareMathOperator{\Lk}{\mathrm Lk}

\DeclareMathOperator{\rk}{\mathrm{rk}}
\DeclareMathOperator{\Sym}{\mathrm{Sym}}

\newcommand{\bm}[1]{H_{#1}^{BM}}

\title{A Division Theorem for Nodal Projective Hypersurfaces}
\date{}
\author{Nikolay Konovalov}
\address{\parbox{\linewidth}{Faculty of Mathematics, HSE University, \\6 Usacheva ulitsa, Moscow 119048, Russia \\
\\
Department of Mathematics, University of Notre Dame, \\255 Hurley Hall, Notre Dame, IN 46556, USA}}
\email{nkonoval@nd.edu, nikolay.konovalov.p@gmail.com}

\begin{document}
\maketitle
\begin{abstract}
Let $V_{n,d}$ be the variety of equations for hypersurfaces of degree $d$ in $\mathbb{P}^n(\mathbb{C})$ with singularities not worse than simple nodes. We prove that the orbit map $G'=SL_{n+1}(\mathbb{C}) \to V_{n,d}$, $g\mapsto g\cdot s_0$, $s_0\in V_{n,d}$ is surjective on the rational cohomology if $n>1$, $d\geq 3$, and $(n,d)\neq (2,3)$. As a result, the Leray-Serre spectral sequence of the map from $V_{n,d}$ to the homotopy quotient $(V_{n,d})_{hG'}$ degenerates at $E_2$, and so does the Leray spectral sequence of the quotient map $V_{n,d}\to V_{n,d}/G'$ provided the geometric quotient $V_{n,d}/G'$ exists. We show that the latter is the case when $d>n+1$.
\end{abstract}

We write $\Pi_{n,d}$ for the space of homogeneous complex polynomials of degree $d$ in $n+1$ variables, 
$U_{n,d}\subset \Pi_{n,d}$ for the affine subvariety of those homogeneous polynomials which give~\emph{smooth} hypersurfaces in $\PP^n(\Co)$, and $\Sigma_{n,d}=\Pi_{n,d}\setminus U_{n,d}$ for its complement. Define for $l=1,\ldots,n+1$
$$\Sigma_{n,d}^{(l)}=\{f\in\Sigma_{n,d}\;|\; V(f)^{\mathrm{sing}}\cap \Lambda^{n-l+1} \neq \emptyset\}.$$
Here, $V(f)\subset \PP^n(\Co)$ is the zero locus of $f$, $V(f)^{\mathrm{sing}}$ is the singular locus of the hypersurface $V(f)$, and $\Lambda^{n-l+1}$ 
is a \emph{fixed} linear subspace in $\PP^n(\Co)$ of codimension $l-1$. In other words, $f\in \Sigma^{(l)}_{n,d}$ if and only if $V(f)$ has at least one singular point in $\Lambda^{n-l+1}$. 
Note that $\Sigma^{(l)}_{n,d}\subset \Sigma_{n,d}$ is an irreducible subvariety of codimension $l$ in $\Pi_{n,d}$, so it defines a fundamental class in the singular Borel-Moore homology: $$[\Sigma^{(l)}_{n,d}]\in \bm{*}(\Sigma_{n,d},\Z).$$ We denote by $\Lk_l \in H^{2l-1}(U_{n,d},\Z)$ the Alexander dual cohomology class of $[\Sigma^{(l)}_{n,d}]$. We point out that $\Lk_l$ does not depend on the choice of $\Lambda^{n-l+1}\subset \PP^n(\Co)$.

Note that the group $G=GL_{n+1}(\Co)$ acts on $\Pi_{n,d}$ by change of variables and this action preserves $U_{n,d}$. We will denote by $U_{n,d}/G$ the geometric quotient of the affine variety $U_{n,d}$ by the $G$-action, which exists if $d\geq 3$ by~\cite[Proposition 4.2]{GITbook} and is affine. (In this paper we use the definition of the geometric quotient given in~\cite{GITbook}.) 

Fix an element $s_0\in U_{n,d}$ and consider the \emph{orbit map} $O\colon G\to U_{n,d}$, $g\mapsto g\cdot s_0$. C.~Peters and J.~Steenbrink~\cite{PS03} proved the following theorem:
\begin{thm}[\cite{PS03}]\label{Division theorem, PS}
The induced map $O^*\colon H^*(U_{n,d},\Q) \to H^*(G,\Q)$ does not depend on $s_0$, and the classes $O^*(\Lk_l)\in H^*(G,\Q),l=1,\ldots, n+1$ are multiplicative generators if $d\geq 3$. In particular, there is an isomorphism of rings:
$$H^*(U_{n,d},\Q) \cong H^*(G,\Q)\otimes H^*(U_{n,d}/G,\Q)$$
compatible with mixed Hodge structures. \qed
\end{thm}

Later this theorem was generalized by A.~Gorinov and the author in~\cite{GK17} to the case of a general reductive group action on the space of~\emph{regular} sections of an equivariant vector bundle over a smooth complex compact variety. In this paper, we study a generalization in a different direction. Namely, we consider a group action not only on the space of regular sections, but on a space of sections with singularities of a certain type.

Let $V_{n,d}$ be the open subvariety of $\Pi_{n,d}$ formed by all homogeneous polynomials $f$ such that the kernel of the Hessian matrix of $f$ at any non-zero singular point $x$ is one-dimensional (or, equivalently, the hypersurface $V(f)$ has no singularities other than simple nodes.) The special linear group $G'=SL_{n+1}(\Co)$ acts on $V_{n,d}$. Fix $s_0\in V_{n,d}$ and consider the orbit map $O'\colon G' \to V_{n,d}$, $g\mapsto g\cdot s_0$. In Theorem~\ref{Division theorem, nodal}, we will show that the induced map of the rational cohomology $${O'}^*\colon H^*(V_{n,d},\Q) \to H^*(G',\Q) $$ is surjective if $n>1,d\geq 3$, and $(n,d)\neq (2,3)$. This implies that there is an isomorphism of rings $$H^*(V_{n,d},\Q)\cong H^*(G',\Q)\otimes H^*_{G'}(V_{n,d},\Q)$$
compatible with mixed Hodge structures, where $H^*_{G'}(V_{n,d},\Q)$ is the equivariant cohomology ring. Finally, we show in Proposition~\ref{proposition:geometric quotient} that a geometric quotient $V_{n,d}/G'$ exists if $d>n+1$, and so in this case we have $H^*_{G'}(V_{n,d},\Q) \cong H^*(V_{n,d}/G',\Q)$, see Corollary~\ref{corollary:geometric quotient}.

We also note that Theorem~\ref{Division theorem, nodal} was used in~\cite{BG22} to compute the cohomology ring $H^*(V_{n,d},\Q)$ if $(n,d)=(2,4)$.

\bigskip

We begin with generalities on spaces of \emph{regular} sections and sections with \emph{nodal singularities}. Let $L$ be a line bundle over a smooth complex projective variety $X$. Let us denote by $J^rL$ the $r$-th jet bundle of $L$, cf.~\cite[Chapter~16.7]{EGAIV}. We recall that $J^0L=L$ and there are maps of vector bundles over $X$: 
\begin{equation}\label{eq:0}
j_r\colon X \times \Gamma(X,L) \to J^rL.
\end{equation}
Moreover, there are short exact sequences:
\begin{equation}\label{eq:1}
0\to \Sym^r(\Omega^1_X)\otimes L \to J^rL \to J^{r-1}L \to 0,
\end{equation}
and the right map is compatible with~\eqref{eq:0}. If $L$ is very ample, then $j_1$ is surjective, cf. \cite[Proposition II.7.3]{Hartshorne}; we denote its kernel by $\widetilde{\Sigma}(L)\subset X \times \Gamma(L)$. By~\eqref{eq:1}, the map $j_2$ restricted to $\widetilde{\Sigma}(L)$ lifts to a map
\begin{equation}\label{eq:2}
\tilde{j}_2\colon \widetilde{\Sigma}(L) \to L\otimes \Sym^2(\Omega^1_X) \hookrightarrow L\otimes \Omega^1_X \otimes \Omega^1_X.
\end{equation}

Consider now the projectivization $\PP(T_X)$ of the tangent bundle $T_X$. We write $p\colon \PP(T_X) \to X$ for the projection map and $\Oh_{\PP(T_X)}(1)$ for the dual of the tautological line bundle $\Oh_{\PP(T_X)}(-1)$ over $\PP(T_X)$, cf.~\cite[Chapter II.7]{Hartshorne}. We have canonical isomorphisms:
$$
p_*(\Oh_{\PP(T_X)}(1))\cong(TX)^*\cong\Omega^1_X,
$$
$$
p_*(p^*(L))\cong L,
$$
and
$$
p_*(p^*(\Omega^1_X))\cong \Omega^1_X.
$$
These isomorphisms and the projection formula give a canonical isomorphism:
$$
L\otimes \Omega^1_X \otimes \Omega^1_X \cong p_*(p^*(L\otimes \Omega^1_X)\otimes \Oh_{\PP(T_X)}(1)).
$$
We now apply the adjunction $p^*\dashv p_*$ to~\eqref{eq:2} and obtain a map:
\begin{equation}\label{eq:3}
h\colon p^*(\widetilde{\Sigma}(L)) \to p^*(L\otimes \Omega^1_X) \otimes \Oh_{\PP(T_X)}(1).
\end{equation}
Informally, $h$ works as follows. If $x\in X$ is a point, then over $x$ we have
\begin{align}
h_x\colon \PP(T_{x,X}) \times \widetilde{\Sigma}_x(L) &\to \PP(T_{x,X}) \times (L_x\otimes \Omega^1_{x,X})\\
h_x([l],s)&=([l],j_2(s)(x)(l)) \label{eq:4}
\end{align}
Here, $l\in T_{x,X}$, $[l]\in \PP(T_{x,X})$ is its equivalence class, and we consider $j_2(s)(x)$ as a map from $T_{x,X}$ to $L_x \otimes \Omega^1_{x,X}$ using $j_1(s)(x)=0$. We note that the twist by $\Oh_{\PP(T_X)}(1)$ is necessary to make~\eqref{eq:4} independent of a choice of a representative $l$ for $[l]$.

Suppose now that $h$ is surjective and denote its kernel by $\widetilde{N}(L)$. Note that $$\varphi\colon\widetilde{N}(L) \to \widetilde{\Sigma}(L)$$ is a proper map between the total spaces of vector bundles. We wish to compute $$\varphi_*\colon\bm{*}(\widetilde{N}(L),\Z) \to \bm{*}(\widetilde{\Sigma}(L),\Z),$$ where $\bm{*}(-,\Z)$ are the singular Borel-Moore homology groups with integer coefficients. Since $\widetilde{N}(L)$ and $\widetilde{\Sigma}(L)$ are oriented vector bundles over $\PP(T_X)$ and $X$ respectively, we have the Thom isomorphisms
\begin{align}
H_*(X,\Z)&\xrightarrow{\sim}\bm{*+2\rk \widetilde{\Sigma}(L)}(\widetilde{\Sigma}(L),\Z), \label{eq:6} \\
H_*(\PP(T_X),\Z) &\xrightarrow{\sim}\bm{*+2\rk \widetilde{N}(L)}(\widetilde{N}(L),\Z) \label{eq:7}.
\end{align}
Here, $\rk$ denotes the \emph{complex} rank of a vector bundle. The next proposition is well known, cf.~\cite{MS74}.
\begin{prop}\label{prop:1}
Under the Thom isomorphisms~\eqref{eq:6}, \eqref{eq:7}, the map
\begin{equation}\label{eq:8}
\varphi_*\colon \bm{*}(\widetilde N(L),\Z) \to \bm{*}(\widetilde{\Sigma}(L),\Z)
\end{equation}
identifies with
\begin{align}
H_*(\PP(T_X),\Z) &\to H_{*-2(r_1-r_2)}(X,\Z) \\
y &\mapsto p_*(y \frown e), \label{eq:9}
\end{align}
where $r_1=\rk\widetilde{\Sigma}(L), r_2=\rk\widetilde{N}(L)$, and $e\in H^{2(r_1-r_2)}(\PP(T_X),\Z)$ is the Euler class of the quotient bundle $$p^*(\widetilde{\Sigma}(L))/\widetilde{N}(L)\cong p^*(L\otimes \Omega^1_X) \otimes \Oh_{\PP(T_X)}(1).\eqno\qed$$ 
\end{prop}

In the sequel, we want~\eqref{eq:9} to be surjective. However, it is more convenient to check that the dual map 
\begin{equation}\label{eq:10}
y \mapsto p^*(y) \smile e, y\in H^*(X,\Z).
\end{equation}
is injective on (singular) cohomology. Recall that the projection $p\colon \PP(T_X) \to X$ is a map between compact oriented manifolds, so we have the pushforward map on cohomology
$$p_!\colon H^*(\PP(T_X),\Z) \to H^{*-2n +2} (X,\Z) $$
defined via the Poincar\'{e} duality. Here $n=\dim X$ is the complex dimension of~$X$.
\begin{prop}\label{prop:2} $p_!(e)= n c_1(L)-2c_1(T_X)$.
\end{prop}

\begin{proof}
Recall that
$$H^*(\PP(T_X))=H^*(X)[c]/(c^n+c_1(T_X)c^{n-1}+\ldots+c_n(T_X)), $$
where $c=c_1(\Oh_{\PP(T_X)}(1))$ is the first Chern class of the line bundle $\Oh_{\PP(T_X)}(1)$. Note that $p_!(c^k)=0$ if $k<n-1$, $p_!(c^{n-1})=1$, and $p_!(c^n)=-c_1(T_X)$.
Set $E=L\otimes \Omega^1_X$, then by the splitting principle we have
$$e=c_n(p^*(E)\otimes \Oh_{\PP(T_X)}(1))=\sum_{i=0}^n c^{n-i} p^*(c_i(E)). $$
Finally, by the projection formula:
\begin{align*}
p_!(e)&=p_!(c^{n}+c^{n-1}p^*c_1(E))=-c_1(T_X)+c_1(\Omega^1_X\otimes L)\\
&=-c_1(T_X)+c_1(\Omega^1_X)+nc_1(L)=-2c_1(T_X)+nc_1(L). \qedhere
\end{align*}
\end{proof}

\begin{cor}\label{cor:1}
Let $X=\PP^n(\Co)$, $L=\Oh(d)$, $n>1$, $d>2$. Then the morphism $h$ (see~\eqref{eq:3}) is a surjective map of vector bundles over $\PP(T_X)$. Moreover, if $(n,d)\neq (2,3)$ and $*<\dim(\widetilde{\Sigma}(L))$, then the map
$$\varphi_*\colon\bm{*}(\widetilde{N}(L),\Q) \to \bm{*}(\widetilde{\Sigma}(L),\Q) $$
is surjective on the rational Borel-Moore homology.
\end{cor}
\begin{proof}
The map $h$ is surjective by a straightforward computation. Indeed, $s\in \Gamma(\PP^n(\Co),\Oh(d))$ is a homogeneous polynomial of degree $d$ in $n+1$ variables. The condition that $j_1(s)([x])=0, [x]\in \PP^n(\Co)$ is equivalent to $x\neq 0$ being a critical point of $s$. Then $j_2(s)([x])$ is the matrix of the second derivatives of $s$ at $x$, i.e. the Hessian matrix of $s$ at $x$. Hence it suffices to show that for each $[x]\in \PP^n(\Co)$ there exists a polynomial $s$ of degree $d$ such that $x$ is a critical point of $s$ and the Hessian matrix of $s$ at $x$ has kernel of dimension $\geq 2$. The latter is clear.

By Proposition~\ref{prop:1} and the projection formula, it is enough to check that $p_!(e)\neq 0$ if $(n,d)\neq(2,3)$. By Proposition~\ref{prop:2}, we have 
$$p_!(e)=(nd-2(n+1))H\in H^2(\PP^n(\Co),\Z), $$
where $H=c_1(\Oh(1))$ is a multiplicative generator of $H^*(\PP^n(\Co),\Z)$. This expression is zero if and only if $(n,d)=(1,4)$ or $(n,d)=(2,3)$.
\end{proof}

Recall that $\Pi_{n,d}$ denotes the space of homogeneous polynomials of degree $d$ in $n+1$ variables $z_0,\ldots,z_n$, i.e. $\Pi_{n,d}=\Gamma(\PP^n(\Co),\Oh(d))$. We let
$$\Sigma_{n,d}=\{f\in \Pi_{n,d} \; | \; f\text{ has a critical point outside } 0\},$$
and we set $N_{n,d}$ to be the subvariety of $\Sigma_{n,d}$ formed by all $f$ such that the kernel of the Hessian matrix of $f$ at some nonzero critical point $x$ contains a $2$-plane $P \ni x$ (or, equivalently, the hypersurface $V(f)$ has a singularity other than a simple node).


\begin{cor}\label{cor:2}
Suppose that $n>1$, $d>2$, and $(n,d)\neq (2,3)$. For $l>1$ there exists a cohomology class $a_l\in H^*(\Pi_{n,d}\setminus N_{n,d},\Q)$ such that $a_l|_{\Pi_{n,d}\setminus \Sigma_{n,d}}=\Lk_l$.
\end{cor}
\begin{proof}
Let $L=\Oh(d)$. We have that
$$\widetilde{\Sigma}(L)=\{(p,f) \in \PP^n(\Co)\times \Pi_{n,d} \; | \; df|_p=0\}.$$ We denote by $\widetilde{\Sigma}^{(l)}(L)$ the subvariety of $\widetilde{\Sigma}(L)$ given by $$\widetilde{\Sigma}^{(l)}(L)=\{(p,f) \in \Lambda^{n-l+1}\times \Pi_{n,d} \; | \; df|_p=0\} \subset \widetilde{\Sigma}(L),$$
where $\Lambda^{n-l+1}$ is the fixed linear subspace of $\PP^n(\Co)$. The projection map 
$$\pi\colon \widetilde{\Sigma}(L) \to \Sigma_{n,d} $$
is proper and generically finite of degree $1$. Moreover, $$[\Sigma_{n,d}^{(l)}]=\pi_*(b_l),$$
where $b_l=[\widetilde{\Sigma}^{(l)}(L)] \in \bm{*}(\widetilde{\Sigma}(L))$. By Corollary~\ref{cor:1}, $b_l=\varphi_*(c_l)$, $c_l\in \bm{*}(\widetilde{N}(L))$ if $l\neq 1$. Finally, let $\pi_1\colon \widetilde{N}(L) \to N_{n,d}$ be the projection map and $\iota\colon N_{n,d} \hookrightarrow \Sigma_{n,d}$ be the embedding. Then $[\Sigma^{(l)}_{n,d}]=\iota_*{\pi_1}_*(c_l)$ if $l>1$ and $a_l$ is the Alexander dual of ${\pi_1}_*(c_l)$.
\end{proof}



Let $G'=SL_{n+1}(\Co)\subset G$ be the special linear group.  Recall that the \emph{universal $G'$-bundle} $EG'$ is a contractible space with a free continuous right $G'$-action such that $EG' \to EG'/G'$  is a locally trivial fiber bundle. We denote by $H^*_{G'}(V_{n,d},\Q)$ the \emph{equivariant cohomology} of $V_{n,d}$, that is the cohomology of the \emph{homotopy quotient} $$(V_{n,d})_{hG'}=EG'\times_{G'} V_{n,d}.$$ 
For a more detailed account of equivariant cohomology we refer the reader e.g.\ to Part I of~\cite{Tu20}. Furthermore, since $G'$ is an algebraic group and $V_{n,d}$ is an algebraic $G'$-variety, we endow equivariant cohomology groups $H^*_{G'}(V_{n,d},\Q)$ with a (functorial) mixed Hodge structure. The construction of this mixed Hodge structure seems to be well-known and it is implicitly contained in~\cite{HodgeIII}; see e.g.~\cite[Proposition~A.5]{BG22} for more details.

There is a fiber sequence 
\begin{equation}\label{eq:11}
G' \to EG'\times V_{n,d} \to  EG'\times_{G'} V_{n,d},
\end{equation}
and the second component of the first map is $O'\colon G' \to V_{n,d}$ given by $g \mapsto g\cdot s_0$ for some $s_0\in V_{n,d}$. Using Theorem~\ref{Division theorem, PS} and Corollary~\ref{cor:2}, we immediately obtain:
\begin{thm}\label{Division theorem, nodal}
Suppose $n>1$, $d\geq 3$, and $(n,d)\neq (2,3)$. Then the classes ${O'}^*(a_l)\in H^*(G',\Q),l=2,\ldots, n+1$ are multiplicative generators. In particular, the Leray-Serre spectral sequence associated with~\eqref{eq:11} degenerates at $E_2$, i.e. there is an isomorphism of rings:
\begin{equation*}
H^*(V_{n,d},\Q) \cong H^*(G',\Q)\otimes H^*_{G'}(V_{n,d},\Q)
\end{equation*}
compatible with mixed Hodge structures. \qed
\end{thm}


\begin{rmk}
The fact that ${O'}^*(\Lk_l)\in H^*(G',\Q),l=2,\ldots, n+1$ are multiplicative generators implies that the connected component $(G'_s)^o$ of the stabilizer subgroup $G'_s\subset G'$ is unipotent for all $s \in V_{n,d}$. Using that one can see that Theorem~\ref{Division theorem, nodal} is false for $(n,d)=(2,3)$. Indeed, the polynomial $s=z_0z_1z_2$ is in $V_{2,3}$ and the stabilizer $G'_s$ contains a torus.
\end{rmk}

\begin{rmk}\label{remark:finite stabilizers}
If a geometric quotient $V_{n,d}/G'$ exists and the quotient map $q\colon V_{n,d} \to V_{n,d}/G'$ is affine, then, under the assumptions of Theorem~\ref{Division theorem, nodal}, the stabilizer subgroups $G'_s, s\in V_{n,d}$ are finite, see~\cite[Theorem~3.1.1 and Proposition~3.1.3]{GK17}. Furthermore, by ibid., Section 4.2.1, the order $|G'_s|$ divides in this case the following number
$$\prod_{i=2}^{n+1}((d-1)^{n+1}+(-1)^{i+1}(d-1)^{n+1-i}). $$
\end{rmk}

\begin{cor}\label{corollary:geometric quotient}
Suppose that $(n,d)$ are as in Theorem~\ref{Division theorem, nodal}, there exists a geometric quotient $V_{n,d}/G'$, and 
the stabilizer subgroups $G'_s, s\in V_{d,n}$ are finite. Then there is an isomorphism: $$H^*_{G'}(V_{n,d},\Q)\cong H^*(V_{n,d}/G',\Q),$$
and the Leray spectral sequence for $q$ degenerates at $E_2$.
\end{cor}

\begin{proof}
The first statement seems to be folklore; for details see e.g. Proposition~A.4 and the remark after Theorem A.3 in~\cite{BG22}. For the second part we again apply Theorem~\ref{Division theorem, PS} and Corollary~\ref{cor:2} to show that the orbit map ${O'}^*$ is a surjective map of the rational cohomology. Thus the Leray spectral sequence for $q$ degenerates at $E_2$ by~\cite[Theorem 2]{PS03} and the example after it, cf.~\cite[Theorem~A.3]{BG22}.
\end{proof}

\begin{rmk}\label{remark:general point}
In Corollary~\ref{corollary:geometric quotient}, the finiteness of the stabilizer subgroups is not a restrictive assumption. Indeed, by the Matsumura-Monsky theorem~\cite[Theorem~1]{MM63}, the stabilizer subgroups $G'_s$ are finite if $s\in U_{n,d}$, $d\geq 3$, and $n\geq 3$. Moreover, if a geometric quotient $V_{n,d}/G'$ exists, then $\dim(G'_s)$ is a constant function on $V_{n,d}$ by the remark before Proposition~0.2 in~\cite{GITbook}. Thus $G'_s,s\in V_{n,d}$ are finite groups if $d\geq 3$, $n\geq 3$, and $V_{n,d}/G'$ exists. 
\end{rmk}

Finally, we show that the geometric quotient $V_{n,d}/G'$ exists if $d>n+1$.

\begin{prop}\label{proposition:geometric quotient}
There exists a geometric quotient $V_{n,d}/G'$ if $d>n+1$. Moreover, $q\colon V_{n,d} \to V_{n,d}/G'$ is affine.
\end{prop}

\begin{proof}
The proof is based on the geometric invariant theory and we will use the terminology and notation from~\cite{GITbook}. Namely, we show that $V_{n,d}$ is contained in the subset of \emph{properly $G'$-stable} points $\Pi_{n,d}^s\subset \Pi_{n,d}$. More precisely, $\Pi_{n,d}^s$ is the preimage of the set of properly $G'$-stable points in the projective space $\PP(\Pi_{n,d})$ with respect to the standard $SL_{n+1}(\Co)$-linearization of $\Oh(1)$.  Then the geometric quotient $\Pi_{n,d}^s/G'$ exists (by Theorem~1.10 in ibid.), and so does $V_{n,d}/G'$. Moreover, the morphism $V_{n,d} \to V_{n,d}/G'$ is affine by ibid., Theorem 1.10(i).

We will show that if $f\in \Pi_{n,d}$ is not properly stable, then the hypersurface $V(f)$ has a singularity worse than a simple node. Suppose that $$f=\sum a_{i_0,\ldots,i_n}z_0^{i_0}\cdots z_n^{i_n} \in \Pi_{n,d}$$ is not a properly stable point. By Theorem~2.1 ibid., there exists a (non-trivial) one-parametric subgroup $\lambda\colon \Co^{*} \to G'$ such that $\mu(f,\lambda)\leq 0$. 
Any one-parametric reductive subgroup of $G'$ is conjugate to a one in the subgroup of diagonal matrices, so we can assume that
$$\lambda(t) =
\begin{pmatrix}
t^{r_0} & 0 & \cdots & 0\\
0 & t^{r_1} & \cdots & 0\\
\vdots & \vdots & \ddots & \vdots \\
0 & 0 & \ldots & t^{r_n}
\end{pmatrix},
$$
where $r_0,\ldots,r_n\in \Z$ and $r_0+\ldots+r_n=0$. By permuting coordinates, we assume further that $r_0\geq r_1\geq \ldots \geq r_n,$ cf.~\cite[Section~4.2, pp. 81-82]{GITbook}.
In this case,
$$\mu(f,\lambda) = \max\{ i_0r_0+\ldots +i_nr_n \;|\; a_{i_0,\ldots, i_n}\neq 0\}. $$

Since $d>n+1$, we have following inequalities for any sequence $r_0\geq\ldots \geq r_n$, $r_0+\ldots +r_n=0$: $$dr_0>0,$$ 
$$(d-1)r_0 +r_i > nr_0+r_i\geq r_0+\ldots+r_n=0, \; \text{if} \; 0<i\leq n; $$
$$(d-2)r_0 +2r_i > (n-1)r_0+r_i+r_{i+1}\geq r_0+\ldots+r_n=0, \; 0<i< n;$$
$$(d-2)r_0 +r_i +r_j > (n-1)r_0+r_i+r_j\geq r_0+\ldots+r_n=0, \; 0<i,j\leq n, \; i\neq j.$$ Since $\mu(f,\lambda)\leq 0$, these inequalities imply that the coefficients of $f$ for $z_0^d$, $z_0^{d-1}z_i$, $z_0^{d-2}z_i^2$ ($i\neq n$), $z_0^{d-2}z_iz_j$ are zero. Therefore, the point $p=(1,0,\ldots,0)$ is a critical point for $f$ and $\dim\ker\mathrm{Hess}_p(f)\geq 2.$
\end{proof}

\begin{rmk}\label{remark:geometric quotient}
It seems likely that the geometric quotient $V_{n,d}/G'$ exists for $3\leq d\leq n+1$ as well; at least this was proven in several particular cases. For instance, $V_{n,d}/G'$ exists if $(n,d)$ is $(3,3)$, $(4,3)$, and $(5,3)$ by~\cite[Proposition~6.5]{Beauville09}, \cite{Allcock03}, and~\cite{Laza09} respectively.
\end{rmk}

\medskip

\noindent {\bf Acknowledgments. }The author is grateful to Aleksandr Berdnikov and Alexey Gorinov for many helpful discussions.

\bibliographystyle{abbrv}
\bibliography{references}

\end{document}